\documentclass[12pt]{article}
\usepackage{amsmath,amsthm,amssymb,amscd,fancybox,epsfig,float,subfigure}
\usepackage{authblk}
\usepackage[margin=1in]{geometry}
\usepackage{graphics}
\usepackage{epstopdf}
\usepackage{multirow}
\usepackage{hhline}
\usepackage{cleveref}
\usepackage{xcolor}
\usepackage{tikz}
\usetikzlibrary{arrows,automata}
\usepackage[nocompress]{cite}
\newtheorem{definition}{Definition}[section]
\newtheorem{remark}{Remark}[section]

\newtheorem{proposition}{Proposition}[section]
\newtheorem{corollary}{Corollary}[section]
\newtheorem{remark1}{Remark}[section]
\newcommand{\beqn}{\begin{equation}}
\newcommand{\eeqn}{\end{equation}}

\newcommand{\cS}{\mathcal{S}}
\newcommand{\sgn}{{\rm sgn}}

\newcommand{\cD}{\mathcal{D}}

\newcommand{\bx}{\boldsymbol{x}}
\newcommand{\bnu}{\boldsymbol{\nu}}

\newcommand{\by}{\boldsymbol{y}}

\newcommand{\bbmat}{\begin{bmatrix}}
\newcommand{\ebmat}{\end{bmatrix}}

\newtheorem{theorem}{Theorem}

\newtheorem{lemma}{Lemma}

\newtheorem{rem}{Remark}

\begin{document}
 \title{On the discretization of Laplace's equation with Neumann boundary conditions on polygonal domains}
\author{Jeremy Hoskins\thanks{Applied Mathematics Program, Yale University, USA. \\ email: jeremy.hoskins@yale.edu}, \,
Manas Rachh\thanks{Center for Computational Mathematics, Flatiron Institute, USA. \\ email: mrachh@flatironinstitute.org}}
 \maketitle
 
 \abstract{In the present paper we describe a class of algorithms for the solution of Laplace's equation on polygonal domains with Neumann boundary conditions. It is well known that in such cases the solutions have singularities near the corners which poses a challenge for many existing methods. 
If the boundary data is smooth on each edge of the polygon, then in the vicinity of each corner the solution to the corresponding boundary integral equation has an expansion in terms of certain (analytically available) singular powers. Using the known behavior of the solution, universal discretizations have been constructed for the solution of the Dirichlet problem. 
However, the leading order behavior of solutions to the Neumann problem is $O(t^{\mu})$ for $\mu \in (-1/2,0)$ depending on the angle at the corner (compared to $O(C+t^{\mu})$ with $\mu>1/2$ for the Dirichlet problem); this presents a significant challenge in the design of universal discretizations. Our approach is based on using the discretization for the Dirichlet problem in order to compute a solution in the ``weak sense'' by solving an adjoint linear system; namely, it can be used to compute inner products with smooth functions accurately, but it cannot be interpolated.  Furthermore we present a procedure to obtain accurate solutions arbitrarily close to the corner, by solving a sequence of small local subproblems in the vicinity of that corner. The results are illustrated with several numerical examples.}

\section{Introduction}
Laplace's equation arises in a vast array of contexts (electrostatics, harmonic functions, low-frequency acoustics, percolation theory, homogenization theory, and the study field enhancements in vacuum insulators for example) and serves as a useful model problem for the study of general elliptic partial differential equations (PDEs). As such, effective numerical methods for quickly and robustly solving Laplace's equation with high accuracy are desirable. Approaches based on potential theory proceed by reducing PDEs to second-kind boundary integral equations (BIEs), where the solution to the boundary value problem is represented by layer potentials on the boundary of the domain. Once these boundary integral equations are discretized the resulting linear systems are better-conditioned than those obtained by directly discretizing the PDE. When the boundary of the domain is smooth there are numerous methods for solving BIEs quickly and accurately (see \cite{hao}, for example). 

Near corners, however, the solutions to both the partial differential equations and corresponding boundary integral equations may have singularities, preventing the application of many traditional methods. Fortunately, a number of approaches have been developed to obviate this difficulty. One class of methods proceeds by introducing many additional degrees of freedom in the vicinity of the corners. In order to prevent the resulting linear systems from becoming intractably large one can use a variety of methods for {\it compressing} the linear system, effectively eliminating the extra degrees of freedom added in the vicinity of the corners. Moreover, the corner refinement and compression can be done in tandem resulting in fast and accurate solvers for elliptic PDEs (see \cite{helsing}, \cite{helsing2},  \cite{ojala}, \cite{helsjcp} and \cite{helsinv}  for  one approach called recursive compressed preconditioning, and  \cite{gillman}, \cite{bremer},\cite{bremer2}, and \cite{bremer3}  for other compression-based methods for solving Laplace's equation). Unfortunately, this approach becomes considerably more expensive in three dimensions limiting its application in that context.

 Another class of methods is based on approximating the solution to the two-dimensional problem by rational functions \cite{gopal2019solving} with poles exponentially clustered near the corners. While this approach allows for fast evaluation of the solution near the boundary of the domain, current implementations are specialized to two-dimensions, and do not scale well  for large problems.
 
 Finally, a recent approach is based on leveraging explicit representations of the solutions to the BIEs in the vicinity of the corner as sums of fractional powers depending on the angle \cite{serkhacha,serkh2016solution}. Using these representations one can construct high-order discretizations which introduce relatively few extra degrees of freedom near the corners (i.e. an amount which is comparable to the number required for smooth portions of the boundary). This approach has been used to generate efficient discretizations for Dirichlet problems for Laplace's equation on polygonal domains\cite{hoskins2019numerical}. 
 
 In this paper we describe a method for solving Laplace's equation on polygonal domains with Neumann boundary conditions given only a discretization of a corresponding Dirichlet problem. Our approach is based on using the discretization of a suitable adjoint problem. In particular, we show that if the transpose of the discretization of a suitable Dirichlet BIE is used, then the resulting solution will be accurate in a ``weak sense''; namely, it can be used to compute inner products with smooth functions accurately, though it cannot be interpolated. We then show how this solution can be used to obtain accurate solutions to the Neumann problem arbitrarily close to a corner by solving a set of local subproblems in the vicinity of that corner.
 
 The paper is organized as follows. In~\cref{sec:mprelim} we review relevant mathematical results associated with Laplace's equation.~\Cref{sec:bvp} describes the reduction of boundary value problems to boundary integral equations via potential theory, and reviews the analytic behavior of solutions near a corner. In~\cref{sec:nprelim,sec:napp} we present our numerical algorithm and the associated analysis. Finally, in~\cref{sec:num} we illustrate its application with several numerical experiments.


\section{Mathematical preliminaries \label{sec:mprelim}}
\subsection{Boundary value problems}
Given a polygonal domain  $\Omega \subset \mathbb{R}^2$ with boundary $\Gamma$ and outward-pointing 
unit normal $\bnu$, as well as a function 
$f: \Gamma \to \mathbb{R},$ we consider 
the following four boundary value problems. 

\begin{enumerate}
\item The interior Dirichlet problem for Laplace's equation:
\begin{align}\label{eqn:BVPIDL}
\Delta u(\bx) &= 0, \hspace{1 cm} \bx \in \Omega,\\
u(\bx) &= f(\bx), \hspace{0.43 cm} \bx \in \Gamma.
\end{align}

\item The exterior Dirichlet problem for Laplace's equation:
\begin{align}\label{eqn:BVPEDL}
\Delta u(\bx) &= 0, \hspace{1 cm} \bx \in \mathbb{R}^2\setminus\Omega,\\
u(\bx) &= f(\bx), \hspace{0.43 cm}\bx \in \Gamma, \\
u(\bx) &= O(1) , \hspace{0.4cm} |\bx| \to \infty.
\end{align}

\item The interior Neumann problem for Laplace's equation:
\begin{align}\label{eqn:BVPINL}
\Delta u(\bx) &= 0, \hspace{1 cm} \bx \in \Omega,\\
\nabla u(\bx)\cdot \bnu(\bx)  &= f(\bx), \hspace{0.45 cm} \bx \in \Gamma,\\
\int_\Gamma f(\bx) {\rm d}S_{\bx} &= 0.
\end{align}

\item The exterior Neumann problem for Laplace's equation:
\begin{align}\label{eqn:BVPEDL}
\Delta u(\bx) &= 0, \hspace{1 cm} \bx \in \mathbb{R}^2\setminus\Omega,\\
\nabla u(\bx)\cdot \bnu(\bx)  &= f(\bx), \hspace{0.43 cm} \bx \in \Gamma, \\
&\hspace{-2 cm}\left|u(\bx) + \left(\frac{1}{2\pi} \int_\Gamma f(\bx) {\rm d}S_{\bx}\right) \log{|\bx|}\right| \to 0 , \hspace{0.4cm} |\bx| \to \infty.
\end{align}

\end{enumerate}

\begin{remark}
The existence and uniqueness of the solutions to the above equations is a well-known result (see \cite{kress1989linear} for example). 
\end{remark}


 \section{Boundary integral equations \label{sec:bvp}}
A classical technique for solving the four Laplace boundary value problems given above is to reduce them to boundary integral equations. Before describing this procedure we first define the single and double layer potential operators and summarize their relevant properties.
 \subsection{Layer potentials}
\begin{definition}\label{def_layerpots}
Given a function $\sigma:\Gamma \to \mathbb{R}$, the single-layer potential is defined by
\beqn
\cS[\sigma](\by) =  \int_{\Gamma} G (\bx,\by)\sigma(\bx) {\rm d}S_{\bx} \, ,
\eeqn
where
\beqn
G(\bx,\by) = -\frac{1}{2\pi} \log{|\bx-\by|}.
\eeqn
Similarly, the double-layer potential is defined via the formula
\beqn
\cD[\sigma](\by) =\int_{\Gamma} \bnu(\bx) \cdot \nabla_{\bx} G(\bx,\by) \sigma(\bx) {\rm d}S_{\bx}.
\eeqn
\end{definition}
In the following we will often refer to the function $\sigma$ as the {\it density} which generates the corresponding potential.
\begin{definition}
For $\bx \in \Gamma$ we define the kernel $K(\bx,\by)$ by
\begin{equation}
K(\bx,\by) =\bnu(\bx) \cdot \nabla_{\bx} G(\bx,\by),
\end{equation}
where $\nu(\bx)$ is the inward-pointing normal to $\Gamma$ at $\bx.$ It will often be convenient to 
work instead with a parametrization of $K.$ In particular, if $\gamma :[0,L] \to \Gamma$ is a counterclockwise 
arclength parametrization of $\Gamma,$ 
we denote by $k:[0,L]\times [0,L] \to \mathbb{R}$ the function defined by
\begin{align}
k(s,t) = K(\gamma(s),\gamma(t)).
\end{align}
\end{definition}

The following theorems describe the behavior of the single and double layer potentials in the vicinity of the boundary curve $\Gamma.$

\begin{theorem}\label{thm:potlim}
Suppose the point $\bx$ approaches a point $\bx_0 = \gamma(t_0)$ (where $\bx_{0}$ is not a corner vertex) from the inside along a path such that 
\begin{align}
-1+\alpha<\frac{\bx-\bx_0}{\|\bx-\bx_0\|} \cdot \gamma'(t_0) <1-\alpha
\end{align}
for some $\alpha >0.$ Then for any continuous function $\sigma:\Gamma \to \mathbb{R},$
\begin{align}
&\lim_{\bx \to \bx_0} \cS[\sigma](\bx) = \cS[\sigma](\bx_0)\\
&\lim_{\bx \to \bx_0} \cD[\sigma](\bx) = \cD[\sigma](\bx_0)- \frac{\sigma(\bx_0)}{2})\\
&\lim_{\bx \to \bx_0} \left.\frac{{\rm d}}{{\rm d}\tau}\right|_{\tau= 0}\cS[\sigma](\bx+\tau \bnu(\bx_0)) = \left.\frac{{\rm d}}{{\rm d}\tau}\right|_{\tau= 0}\cS[\sigma](\bx_0+\tau \bnu(\bx_0))+ \frac{\sigma(\bx_0)}{2}.
\end{align}
Similarly, if $\bx$ approaches a point $\bx_0 = \gamma(t_0)$ from the outside then for any continuous function $\sigma:\Gamma \to \mathbb{R},$
\begin{align}
&\lim_{\bx \to \bx_0} \cS[\sigma](\bx) = \cS[\sigma](\bx_0)\\
&\lim_{\bx \to \bx_0} \cD[\sigma](\bx) = \cD[\sigma](\bx_0)+ \frac{\sigma(\bx_0)}{2}\\
&\lim_{\bx \to \bx_0} \left.\frac{{\rm d}}{{\rm d}\tau}\right|_{\tau= 0}\cS[\sigma](\bx+\tau \bnu(\bx_0)) = \left.\frac{{\rm d}}{{\rm d}\tau}\right|_{\tau= 0}\cS[\sigma](\bx_0+\tau \bnu(\bx_0))- \frac{\sigma(\bx_0)}{2}.
\end{align}
\end{theorem}

Next we define the following operator which arises in the study of Neumann boundary value problems.
\begin{definition}\label{def_singder}
Let $\cS$ be the single-layer potential operator and $\bnu\cdot \nabla \cS$  denote its normal derivative restricted to $\Gamma.$ In particular, for $\bx_0 \in \Gamma,$
\begin{align}
\bnu(\by)\cdot \nabla \cS[\rho] (\bx_0) = \left.\frac{{\rm d}}{{\rm d}\tau}\right|_{\tau = 0} \cS[\rho](\bx_0 + \tau \bnu(\bx_0)),
\end{align}
where $\gamma(t_0) = \bx_0.$
\end{definition}

The following proposition relates the normal derivative of the single-layer operator to the double-layer operator. Its proof follows directly from Definitions \ref{def_layerpots} and \ref{def_singder}.
\begin{proposition}
Let $\cS,\cD:L^2(\Gamma) \to L^2(\Gamma)$ be defined as above. Let $\bnu \cdot \nabla \cS$ denote the normal derivative of $\cS$ in the sense of the previous definition. Then $\bnu \cdot \nabla \cS= \cD^{T}$ where $T$ denotes the adjoint operator with respect to the inner product
\beqn
\langle f,g \rangle = \int_{0}^{L} f(\gamma(t)) g(\gamma(t)) {\rm d}t\,,
\eeqn
where $\gamma:[0,L] \to \Gamma$ is a counterclockwise arclength parametrization of $\Gamma.$ 
In particular, for all $\rho,\sigma \in L^2(\Gamma),$
\begin{align}
\cD[\sigma](\gamma(t)) = \int_0^L k(s,t)\,\sigma(\gamma(s))\, {\rm d}s
\end{align}
and
\begin{align}
\bnu(\gamma(t))\cdot \nabla\cS[\rho](\gamma(t)) = \int_0^L k(t,s)\,\rho(\gamma(s))\, {\rm d}s.
\end{align}
\end{proposition}

\subsection{Reduction of boundary value problems}
In this section we describe the conversion of the Laplace boundary value problems (interior Dirichlet, exterior  Dirichlet, interior Neumann, and exterior Neumann) to second-kind integral equations.

\begin{theorem} [Interior Dirichlet problem for Laplace's equation]
For every $f:[0,L] \rightarrow \mathbb{R}$ in $L^{2}[0,L],$ there exists a unique $\sigma \in L^{2}[0,L]$ which satisfies
\begin{align}
f(s)= - \frac{\sigma(s)}{2}+ \int_0^L k(t,s)\,\sigma(t)\,{\rm d}t,
\label{bie_intd}
\end{align}
Moreover, the solution to the interior Dirichlet problem for Laplace's equation with boundary data $f$ is given by $u(\by) = \cD[\sigma](\by)$ for all $\by \in \Omega.$
\end{theorem}

\begin{theorem} [Exterior Dirichlet problem for Laplace's equation]
For every $f:[0,L] \rightarrow \mathbb{R}$ in $L^{2}[0,L]$ there exists a unique $\sigma \in L^{2}[0,L]$ which
satisfies
\begin{align}
f(s)=  \frac{\sigma(s)}{2}+ \int_0^L (k(t,s) + 1)\,\sigma(t)\,{\rm d}t,
\end{align}\label{bie_extd}
for all $s \in [0,L].$ 
Moreover, the solution to the exterior Dirichlet problem for Laplace's equation with boundary data $f$ is given by $u(\by) = \cD[\sigma](\by) + \int_{0}^{L} \sigma(t) dt$
for all $\by \in \mathbb{R}^2 \setminus \Omega.$
\end{theorem}

\begin{theorem}[Interior Neumann problem for Laplace's equation]
For every  $f:[0,L] \rightarrow \mathbb{C}$ in $L^{2}[0,L]$ such that $\int_{0}^{L} f(t) = 0,$ there exists a unique $\sigma \in L^{2}[0,L]$ which satisfies
\begin{align}
f(s)=  \frac{\sigma(s)}{2}+ \int_0^L (k(s,t) + 1)\,\sigma(t)\,{\rm d}t,
\label{bie_intn}
\end{align}
Moreover, the solution to the interior Neumann problem for Laplace's equation with boundary data $f$ is given by $u(\by) = \cS[\sigma](\by)$ for all $\by \in \Omega.$
\end{theorem}

\begin{theorem}[Exterior Neumann problem for Laplace's equation]
For every $f:[0,L] \rightarrow \mathbb{C}$ in $L^{2}[0,L]$ there exists a unique $\sigma \in L^{2}[0,L]$ which satisfies
\begin{align}
f(s)=  \frac{\sigma(s)}{2}+ \int_0^L k(s,t) \,\sigma(t)\,{\rm d}t,
\label{bie_extn}
\end{align}
Moreover, the solution to the interior Neumann problem for Laplace's equation with boundary data $f$ is given by $u(\by) = \cS[\sigma](\by)$ for all $\by \in \Omega.$
\end{theorem}

\subsection{Corner expansions}

In the remainder of this section, we assume $\Gamma$ is an open wedge with sides of length one and interior angle $\pi \alpha$ with $0 <\alpha<2.$ Let $\gamma:[-1,1] \rightarrow \Gamma$ be an arc length parametrization of $\Gamma$ and $\nu:[-1,1] \rightarrow \mathbb{R}^2$ be the inward-pointing normal to $\Gamma.$ 
The following theorem  gives an explicit representation of the solutions of the boundary integral equation (\ref{bie_intd}) in this geometry.

\begin{theorem}[\cite{serkhacha}]\label{thm_cord}
Suppose that $0<\alpha<2$ and that $N$ is a positive integer. Let 
$\lceil \cdot \rceil$ and $\lfloor \cdot \rfloor$ denote the ceiling and floor functions,
respectively, and define $\overline{L},$ $\underline{L},$ $\overline{M},$ and 
$\underline{M}$ by the following formulas
\begin{align}
&\overline{L} = \left\lceil\frac{\alpha N}{2} \right \rceil, \\
&\underline{L} = \left\lfloor\frac{\alpha N}{2} \right \rfloor,\\
&\overline{M} = \left\lceil\frac{(2-\alpha) N}{2} \right \rceil, \\
&\underline{M} = \left\lfloor\frac{(2-\alpha) N}{2} \right \rfloor.
\end{align}
Suppose further that $\sigma$ is defined via the formula
\begin{align}
\sigma(t) =& b_0+\sum_{i=1}^{\overline{L}} b_i |t|^{\frac{2i-1}{\alpha}}+\sum_{i=1}^{\underline{M}} b_{\overline{L}+i} |t|^\frac{2i}{2-\alpha}\left(\log|t|\right)^{\sigma_{N,\alpha}(i)}\nonumber\\
&+\sum_{i=1}^{\overline{M}}c_i \sgn(t) |t|^\frac{2i-1}{2-\alpha}+\sum_{i=1}^{\underline{L}}c_{\overline{M}+i} \sgn(t) |t|^\frac{2i}{\alpha}\left(\log|t|\right)^{\nu_{N,\alpha}(i)}\label{eqn_rhfm}
\end{align}
where $b_0,b_1,\dots,b_N$ and $c_1,c_2,\dots,c_N$ are arbitrary real numbers and the functions $\sigma_{\alpha,N}(i)$ and $\nu_{\alpha,N}(i)$ are defined as follows
\begin{align}
\sigma_{N,\alpha}(i) = \begin{cases}
1 \quad\quad & \,\text{if}\,\, \frac{2i}{2-\alpha} = \frac{2j-1}{\alpha}\,\,\text{for some}\,\, j \in \mathbb{Z},\, 1\le j \le \left\lceil \frac{\alpha N}{2} \right\rceil\\
0 \quad\quad & \,\text{otherwise},
\end{cases}
\end{align}
\begin{align}
\nu_{N,\alpha}(i) = \begin{cases}
1 \quad\quad & \,\text{if}\,\, \frac{2i}{\alpha} = \frac{2j-1}{2-\alpha}\,\,\text{for some}\,\, j \in \mathbb{Z},\, 1\le j \le \left\lceil \frac{(2-\alpha) N}{2} \right\rceil\\
0 \quad\quad & \,\text{otherwise}.
\end{cases}
\end{align}
\noindent If $f$ is defined by
\begin{align}
f(t) = -\frac{\sigma(s)}{2} +\int_{-1}^1 k(t,s)  \sigma(t)\,{\rm d}t.
\label{eqn_grho}
\end{align}
and $\sigma$ is defined by (\ref{eqn_rhfm}) then there exist two sequences of real numbers $\beta_0$, $\beta_1,\,\dots$ and $\gamma_0,$ $\gamma_1,\,\dots$ such that
\begin{align}
f(t) = \sum_{n=0}^\infty\beta_n |t|^n + \sum_{n=0}^\infty \gamma_n \sgn(t) |t|^n,
\label{eqn_gfrm}
\end{align}
for all $-1 \le t \le 1.$ Conversely, suppose that $f$ has the form (\ref{eqn_gfrm}). Suppose further that $N$ is an arbitrary positive integer. Then, for all angles $\pi \alpha$ there exist unique real numbers $b_0,\,b_1,\,\dots,\,b_N$ and $c_0,\,c_1,\,\dots,\,c_N$ such that $\rho,$ defined by (\ref{eqn_rhfm}), solves equation  (\ref{eqn_grho}) to within an error ${\rm O}(t^{N+1}).$ 
\end{theorem}

\begin{remark1}
A similar result holds for the case where the identity term in (\ref{eqn_grho}) is replaced by its negative; the change in sign corresponds to replacing the boundary integral equation for the interior Dirichlet problem with the boundary integral equation corresponding to exterior Dirichlet problem. Similar expansions also hold for both the exterior and interior Neumann problems, in which case the singular powers are obtained by subtracting one from the singular powers arising in the Dirichlet problem.
\end{remark1}

The following corollary, proved in \cite{serkhacha} gives a characterization of the solutions to the Dirichlet and Neumann boundary integral equations in the vicinity of a corner.

\begin{corollary}
Let $\Gamma$ be the boundary of a polygonal region and suppose one of its corners has interior angle $\pi \alpha$ where $\alpha \in (0,2).$ Let $\gamma:(-\delta,\delta) \rightarrow \mathbb{R}^2$ be an arclength parametrization of $\Gamma$ in the vicinity of the corner, with $\gamma(0)$ coinciding with the corner. If the boundary data, $f,$ is analytic on either side of the corner then there exist unique real numbers $b_0,\,b_1,\,\dots,\,b_N$ and $c_0,\,c_1,\,\dots,\,c_N$ such that the density, $\rho,$ defined by (\ref{eqn_rhfm}) satisfies the interior Dirichlet boundary integral equation to within an error ${\rm O}(t^{N+1})$ for $t$ within $\delta$ of the corner. For the Neumann problems the representation is the same with the powers in the expansion reduced by one.
\end{corollary}



\section{Numerical preliminaries \label{sec:nprelim}}
In this section we summarize the numerical tools which are necessary for the main result. In particular we summarize the method for discretizing the boundary integral equation for the Dirichlet problem described in \cite{hoskins2019numerical}, which uses the expansion in Theorem \ref{thm_cord}.
\subsection{Discretization of the Dirichlet problem}\label{sec:disc_dir}
In this section we sketch an algorithm for solving the interior Dirichlet boundary integral equation using a Nystr\"{o}m method; the exterior Dirichlet boundary integral equation can be discretized in a similar way. See \cite{hoskins2019numerical} for a thorough description of the method. 

The Nystr\"{o}m method proceeds as follows. We begin by constructing a discretization of the boundary $\Gamma$ with nodes $s_1,\dots,s_N,$ and weights $w_1,\dots,w_N,$ which enable interpolation of the left- and right-hand sides of the boundary integral equation
\begin{align}\label{eqn:numpreint}
f(s) = -\frac{\sigma(s)}{2} + \int_0^L k(t,s) \,\sigma(t)\, {\rm d} t
\end{align}
with precision $\epsilon.$ In other words, given $f(s_i),-\sigma(s_i) + \int_0^L k(t,s_i) \,\sigma(t)\, {\rm d} t,$ for $i=1,\dots,N,$ the values $f(s)$ and $-\sigma(s)/2 + \int_0^L k(t,s) \,\sigma(t)\, {\rm d} t$ can be obtained for all $0 \le s \le L$ to within $\epsilon.$

Once these nodes and weights have been generated we proceed by enforcing equality of (\ref{eqn:numpreint}) at the discretization nodes, which yields the system of equations
\begin{align}
f(s_i) \sqrt{w_i} = -\frac{\sigma(s_i)\sqrt{w_i}}{2} +  \sqrt{w_i}\int_0^L k(t,s_i)\, \sigma(t)\,{\rm d}t,\quad i=1,\dots,N. \label{eqn_1std}
\end{align}
We note that scaling by the square root of the weights in the above equation is equivalent to solving the problem in the $L^2$ sense, and results in discretized operators with condition numbers which are close to those of the original physical systems \cite{bremer3}. The new unknowns are $\sigma_i = \sigma(s_i) \sqrt{w_i},$ $i=1,\dots,N.$ Next, for each interpolation node $s_i$ we find a collection of weights $W_{ij}$ such that
\begin{align}
\left|\int_0^L k(t,s_i)\, \sigma(t)\,{\rm d}t- \sum_{j=1}^N W_{ij} \sigma_j\, \sqrt{w_j} \right| < \epsilon,
\end{align}
resulting in the linear system
\begin{align}
  -\frac{\sigma_i}{2} +  \sum_{j=1}^N W_{ij} \sigma_j \sqrt{w_i w_j}=f(s_i) \sqrt{w_i},\quad i,j=1,\dots,N. 
\end{align}


\subsubsection{Obtaining interpolation nodes}

The boundary $\Gamma$ is separated into a collection of intervals which are at least a fixed distance $\delta$  (measured in terms of arclength) away from a corner and the collection of intervals of length $2 \delta$ centered about each corner. The former are discretized using a standard smooth quadrature rule such as nested Gauss-Legendre quadrature while the latter are discretized using a custom set of interpolation nodes constructed in the following way.

First, all functions of the form $x^\mu,$ $\mu \in \{0\} \cup [1/2,50],$ $x \in [0,1]$ are discretized using nested Gauss-Legendre panels in $x$ and a single Gauss-Legendre panel in $\mu$. This creates a $N \times M$ matrix where $N$ denotes the number of spatial discretization points $r_i$ and $M$ denotes the number of $\mu_j$ chosen. $M$ and $N$ are increased until it is guaranteed that using Lagrange interpolation from the nested discretization the function $x^\mu$ can be interpolated to within an $L^2$ error less than $\epsilon$ on the interval $[0,1]$ for any $\mu$ in the specified range. A singular value decomposition is then performed on the $N\times M$ matrix. Let $K$ denote the number of singular values greater than $\epsilon.$ The right singular vectors correspond to discretizations of an orthonormal set of functions $\phi_1,\dots,\phi_K$ such that $x^\mu$ is in the span of $\phi_1,\dots,\phi_K$ to within an accuracy of $\epsilon.$

Finally, a set of interpolation points $x_j,$ $j=1,\dots,K$ and quadrature weights $w_j,$ $j=1,\dots,K$ are chosen for $\phi_1,\dots,\phi_K$ such that the matrix $U_{ij} = \phi_i(x_j) \sqrt{w_j}$ is well-conditioned. In practice suitable interpolation points can be obtained by using the roots of $\phi_{K+1}$ and calculating the corresponding weights by solving a linear system. The corresponding discretization nodes and weights for the corner-containing intervals of $\Gamma$ are obtained by suitable translations and scalings of $\{x_j\}$ and $\{w_j\}.$

\subsubsection{Construction of quadrature rules}

Once the discretization has been constructed it is necessary to construct appropriate quadrature for the integrals appearing in equation (\ref{eqn_1std}). When $s_i$ and $t$ do not belong to the same corner panel (in particular when either is not itself contained in a corner panel) then the weights and nodes associated with the discretization can be used as the quadrature rule. When $s_i$ corresponds to a corner panel special care must be taken. Instead, using an algorithm for generating generalized Gaussian quadratures \cite{bremer2010}, quadrature nodes are chosen which integrate
\begin{align}
\int_{0}^\delta k(t,s_j) \tilde{\phi}_j(t)\,{\rm d}t
\end{align} 
where $\tilde{\phi}_j$ is a suitably scaled and translated copy of the singular function obtained in the discretization step, and for ease of exposition we assume that the corner panel corresponds to $(-\delta,\delta)$ in the parametrization with $t=0$ corresponding to the corner itself. Moreover, in light of symmetry between the two legs of the wedge it suffices to design quadratures assuming  $s_j$ lies in the half of a corner panel parametrized by $(-\delta,0).$
\begin{remark1}
Due to scale invariance, it suffices to compute quadratures for 
\begin{align}
\int_{0}^1 k(t,-x_j) {\phi}_j(t)\,{\rm d}t,
\end{align} 
where $x_j$ was one of the original discretization nodes generated on the interval $[0,1].$
\end{remark1}
\begin{remark1}
By interpolating from the discretization nodes to these quadrature nodes we obtain a set of weights $\tilde{W}_{i,j}$ such that if $s_1,\dots,s_{2K}$ correspond to the discretization of a corner parametrized by $(-\delta,\delta)$ with $0$ corresponding to the corner then
\begin{align}
\left| \int_{-\delta}^\delta k(t,s_i) \, \tilde{\phi}_m(t)\,{\rm d}t - \sum_{j=1}^{2K} \tilde{W}_{ij} \tilde{\phi}_{m}(t_j) \right| < \epsilon
\end{align}
for all $i=1,\dots,2K$ and $m=1,\dots,K.$
\end{remark1}

After all the quadratures have been constructed the result is an $N \times N$ linear system the solution of which gives an approximation to $\sigma$ sampled at the discretization nodes. 

\begin{definition}\label{def:seps}
Let $S_\epsilon \subset L^2([0,L])$ denote the set of functions which can be interpolated from their values at the $N$ discretization nodes to any point in $[0,L]$ with a relative $L^2$ accuracy of $\epsilon.$ That is to say that for $f \in S_\epsilon$ if $\tilde{f}:[0,L] \to \mathbb{R}$ denotes the function obtained by interpolating using the values $f(s_1),\dots,f(s_N)$ then $\|f -\tilde{f}\|_{L^2} < \epsilon.$ 
\end{definition}

The results of this algorithm are summarized in the following theorem (see \cite{}).
\begin{theorem}
Let $A$ be the $N\times N$ matrix obtained by discretizing the interior Dirichlet problem in the preceding manner. In particular if $f \in S_\epsilon $ is piecewise analytic and ${\bf f} = (\sqrt{w_1}f(s_1),\dots,\sqrt{w_N}f(s_N))^T$ then
\begin{align}
\underline{\bf \sigma} = A^{-1} {\bf f}
\end{align}
can be interpolated to a function $\tilde{\sigma}$ which is within $\epsilon$ of the true density $\sigma$ in an $L^2$-sense. 
\end{theorem}

\subsection{Discretization of the Neumann problem}
In principle a similar method could be employed to discretize the Neumann boundary integral equations. Unfortunately, the singular nature of the powers (the smallest in the expansion given in Theorem \ref{thm_cord} lies in the range $(-1/2,0)$) makes it difficult to produce universal discretizations and quadratures which work for large ranges of angles. When the above method is run on these problems, discretization nodes tend to accumulate close to the corner (within $10^{-14}$). Apart from posing certain numerical challenges, it also makes the task of finding suitable quadrature formulae difficult. Instead, a different set of discretization nodes and a different set of quadrature nodes can be constructed for each angle, though this would significantly increase the precomputation cost.

Finally, in many applications one already has a discretization of the Dirichlet problem. For example, when considering Laplace transmission problems or triple junction problems one has to solve two decoupled boundary integral equations: one of them a Dirichlet-type boundary integral equation with the diagonal term scaled and the other a Neumann-type boundary integral equation with the identity term scaled (see \cite{hoskins2018numerical} and \cite{hoskins2019solution} for example). In such cases it is convenient to reuse the Dirichlet discretization for the Neumann problem.  



\section{Numerical apparatus \label{sec:napp}}
\subsection{Adjoint discretization}

The following lemma relates the discretization of the inverse of an operator to the adjoint of the discretization of its inverse. Its proof follows directly from the definition of the adjoint and is omitted.
\begin{lemma}\label{lem_adjm}
Suppose $A: L^2([0,L]) \to L^2([0,L])$ is a bounded invertible operator and that $A_\epsilon$ is an operator such that
\begin{align}
\left| \langle f,A^{-1} g \rangle - \langle f,A^{-1}_\epsilon g \rangle \right| \le \epsilon \|f\| \|g\|,
\end{align}
for all $f$ and $g$ in some subspace $S_{\epsilon} \subset L^2([0,L]).$ Here $\langle \cdot, \cdot \rangle$ denotes the inner product on $L^2([0,L])$ and $\| \cdot\|$ denotes the norm for $L^2([0,L]).$ Then, for all functions $f$ and $g$ in $S_{\epsilon}$
\begin{align}
\left| \langle f,(A^{-1})^* g \rangle - \langle f,\left(A^{-1}_\epsilon\right)^* g \rangle \right| \le \epsilon \|f\| \|g\|
\end{align}
where $\,{{}^*}\,$ denotes the adjoint.
\end{lemma}

The following corollary follows immediately from the previous result.
\begin{corollary}
Let $A$ be the $N\times N$ matrix obtained by discretizing the interior Dirichlet problem and $S_\epsilon$ be the collection of functions given by Definition \ref{def:seps}. Then for all functions $f,g \in S_\epsilon$
\begin{align}
\left| \langle {\bf g}, (A^T)^{-1}{\bf f} \rangle - \int_0^L g(t)\,\sigma(t)\,{\rm d}t \right| < \epsilon \|f\|\, \|g\|,
\end{align}
where ${\bf f},{\bf g}$ are the discretizations of $f$ and $g$ scaled by the square roots of the discretization weights, and $\sigma$ is the solution to the exterior Neumann problem with boundary data $f$.
\end{corollary}

Hence a discretization of the Neumann problem can be obtained simply by taking the adjoint of the Dirichlet problem. The resulting density $\sigma$ obtained is accurate in a weak sense, ie. its inner products against functions in $S_\epsilon$ are accurate to within an error of $\epsilon.$

We conclude this section with a few remarks. 

\begin{remark1}
\label{rem:far-field-accuracy}
We observe that if the solution to the boundary value problem is being calculated at a point $\by \in \mathbb{R}^2\setminus \Omega$ more than one panel length away from the boundary curve $\Gamma$ then the Neumann density $\sigma$ obtained using the above result will give an accuracy of $\epsilon,$ ie. the function $K(\by,\gamma(t)) \in S_\epsilon.$ Thus accurate values of the solution in the far-field can be obtained almost immediately.
\end{remark1}
\begin{remark1}
\label{rem:strong-sol-smooth}
Similarly, if the point $\by \in \mathbb{R}^2 \setminus \Omega$ at which the solution to the Neumann boundary value problem is to be calculated lies close to a smooth panel then the density $\sigma$ near that point can be interpolated to a finer set of quadrature points and the value of $u(\by)$ can once again be obtained to precision $\epsilon.$ We note, however, that in general the density in the vicinity of a corner cannot be interpolated accurately. This follows from the fact that the interpolation scheme constructed is only guaranteed to interpolate the powers arising in the Dirichlet problem accurately near the corner. The collection of singular powers arising in Neumann problems contain negative powers which are not contained in this set and hence are not interpolated accurately.
\end{remark1}

\subsection{Weak corner re-solving \label{sec:resolve}}
In this section we address the problem highlighted in the previous one; namely, the accurate evaluation of the solution to the exterior Neumann problem in the vicinity of a corner. Our approach is based on the observation that the potential generated by the density on the boundary outside of a sufficiently small neighborhood of the corner is smooth when evaluated in the vicinity of the corner. This allows us to convert the problem of evaluating the potential near the corner (given the approximation to the density obtained using the adjoint approach described in the previous section) into a purely local one. In particular, we re-discretize only a small neighborhood of the corner which in turn allows us to evaluate the potential arbitrarily close to the corner to within a small factor of machine precision. 

In the following we assume that we are given a discretization of the interior Dirichlet boundary integral equation (\ref{bie_intd}) with nodes $x_1,\dots,x_N$ and corresponding weights $w_1,\dots,w_N.$ In particular, we assume that the discretization nodes are obtained by subdividing the boundary into panels. Those panels which contain a vertex are discretized using a custom discretization scheme (see Section \ref{sec:disc_dir}) while the remaining panels are discretized using a standard smooth quadrature rule (such as Gauss-Legendre or Chebyshev nodes). In the following we assume that an $M$-point Gauss-Legendre quadrature rule is used and the corner panels are discretized using $P$ nodes (together with a collection of orthonormal functions on that interval $\phi_1,\dots,\phi_P$).

Additionally, we denote the discretization of the interior Dirichlet operator (using the custom quadratures described in Section \ref{sec:disc_dir}) by $A.$ Let $\underline{f} = (f_1,\dots,f_N)^T$ where $f_i = f(x_i) \sqrt{w_i}$ and $f: \partial \Omega \to \mathbb{R}$ is the right-hand side of the exterior Neumann problem. Finally, let $\underline{\sigma}$ be the approximation to the density (scaled by the square roots of the weights) obtained by solving the linear system
\begin{align}\label{eqn:adj_linalg}
A^T \underline{\sigma} = \underline{f}.
\end{align}
For notational convenience we let $\gamma:[-\delta,L-\delta] \to \partial \Omega$ be a counterclockwise arclength parametrization of $\partial \Omega$ such that  $\gamma(0)$ corresponds to a vertex and $\gamma[-\delta,\delta]$ corresponds to a corner panel.

For a panel $\gamma([s_1,s_2])$ with discretization nodes $x_i,\dots,x_{i+M}$ corresponding to a Gauss-Legendre panel the density is smooth and thus it is expected to be well-represented in the basis of Legendre polynomials (shifted and scaled to the interval $[s_1,s_2]$). Hence standard interpolation techniques can be used to obtain an accurate approximation to the density $\sigma$ on the interval $s_1 \le s\le s_2$ . Typically we use 16th order Gauss-Legendre panels and choose their sizes so that their length is no more than their distance to the nearest corner. This latter choice guarantees that for any $\epsilon>0$ there exists an $M$ such that if the Gauss-Legendre panels are discretized using an $M$-point Gauss-Legendre rule then the density on that panel can be interpolated to relative precision $\epsilon$ in an $L^2$-sense.
(We discuss a sketch of a proof in~\cref{sec:appb})

For corner panels the nodes were constructed to enable stable interpolation of densities $s^\mu,$ $\mu \in 0 \cup [1/2,50],$ on the interval $s \in (-\delta,\delta)$ - assuming for simplicity that the corner is at $0$ and the panel is of length $2 \delta.$ As mentioned above, the density is expected to contain terms of the form $s^\mu$ for some finite collection of $\mu$ in the interval $(-1/2,1/2),$ and hence will not in general be stably interpolable on the interval $(-\delta,\delta).$ However, it is possible to use the density obtained using (\ref{eqn:adj_linalg}) to construct a sequence of nested problems in the neighborhood of the corner, the solutions of which enable accurate interpolation of the density arbitrarily close to the vertex. The number of these problems depends only on the distance of the closest evaluation point to the corner. In particular, if $r$ is the smallest distance of an evaluation point from the corner then only $\log_2 r/\delta$ levels are required. Each problem involves the solution of a small linear system (typically less than $100 \times 100$) and as such can be performed quickly. Furthermore, we note that the algorithm can be easily parallelized to treat multiple corners concomitantly.

We begin with the following proposition, the proof of which follows immediately from the definition of the kernel $k$ and is omitted.
\begin{proposition}
Suppose that $f$ be a piecewise-analytic function in $S_\epsilon$ and  $\underline{\sigma} = (A^T)^{-1} {\bf f}$ is the approximation to the Neumann density obtained using the adjoint of the discretization for the interior Dirichlet boundary integral equation. Further suppose that the discretization nodes are ordered so that $s_1,\dots,s_P$ correspond to the corner panel associated with the interval $(-\delta,\delta),$ $s_{P+1},\dots,s_{P+M}$ correspond to the Gauss-Legendre panel immediately to the left associated with the interval $(-2\delta,-\delta),$ and $s_{P+M+1},\dots,s_{P+2M}$ to the Gauss-Legendre panel immediately to the right associated with the interval $(\delta,2\delta).$
Then
\begin{align}
h(t) = \sum_{i=P+2M+1}^N k(s_i,t) \sqrt{w_i} \,\underline{\sigma}_i
\end{align}
is an analytic function of $t$ for all $t \in (-2\delta,2\delta).$
\end{proposition}

In light of this we consider the following integral equation
\begin{align}\label{eqn:resolve_neum0}
-\sigma(s) + \int_{-2\delta}^{2\delta} k(s,t) \sigma(t)\,{\rm d}t = f(s)-h(s), \quad -2\delta \le s \le 2\delta.
\end{align}
We note that the solution to (\ref{eqn:resolve_neum0}) is equal to the solution of the original boundary integral equation (\ref{bie_extn}) restricted to the interval $[-2\delta,2\delta].$ Taking the adjoint of (\ref{eqn:resolve_neum0}) we obtain
\begin{align}\label{eqn:resolve_diri0}
-\sigma(s) + \int_{-2\delta}^{2\delta} k(t,s) \sigma(t)\,{\rm d}t = f(s)-h(s), \quad -2\delta \le s \le 2\delta.
\end{align}
which is a Dirichlet boundary integral equation for a wedge with a piecewise analytic right-hand side. In particular, we can discretize the operator using the method summarized in the previous section. Specifically, we subdivide the interval $[-\delta,\delta]$ into three subintervals $I_0 = [-\delta,\delta/2],$ $L_0 = [-\delta/2,\delta/2],$ and $J_0 = [\delta/2,\delta].$ On $I_0$ and $J_0$ we place standard Gauss-Legendre discretization nodes, while on $L_0$ we use the custom discretization scheme for corners, outlined in Section \ref{sec:disc_dir} (see \cite{} for a detailed description of the method). On the intervals $[-2\delta,-\delta]$ and $[\delta,2\delta]$ we use the same discretization nodes and weights as in the original system for those intervals (we call these panels $K_0$ and $Q_0$ respectively). Let $\underline{f}_0$ denote the right-hand side of (\ref{eqn:resolve_neum0}) evaluated at these discretization nodes and scaled by the square roots of the corresponding weights. Let $A_0$ be the discretization of the interior Dirichlet problem operator (ie. the operator acting on $\sigma$ on the left-hand side of (\ref{eqn:resolve_diri0})). We note that due to the scale invariance of Laplace's equation for polygonal domains the portion of $A_0$ corresponding to the self-interaction of $L_0$ is a submatrix of the original matrix $A.$ All other blocks can be generated using the discretization nodes as quadrature nodes.

The analysis of the previous section then shows that if $\underline{\sigma}_0$ is the solution of the equation
\begin{align}
A_0^T \underline{\sigma}_0 = \underline{f}_0
\end{align}
then $\underline{\sigma}_0$ gives a {weak solution} to the integral equation (\ref{eqn:resolve_neum0}), i.e. for any function $g$ which is analytic on $[-2\delta,0]$ and $[0,2\delta]$ the inner product $\langle g,\sigma \rangle$ can be calculated to precision $\epsilon$ using the solution $\underline{\sigma}_0.$ Moreover, since the true density $\sigma$ is smooth on $[\delta,2\delta]$ and $[-2\delta,-\delta]$ the Gauss-Legendre discretization allows accurate interpolation of the density on those regions.

\begin{remark}
Though the above method produces a viable method for reducing the problem, as written the reduction is non-local --- in order to compute the right-hand side for the sub-problem one must evaluate contributions from the rest of the domain.
\end{remark}

The following theorem shows that the right-hand side $\underline{f}_0$ can be computed only using local data (i.e. values of the weak solution in the vicinity of the corner).\\
\begin{theorem}
Suppose that $\underline{f}_0$ is the discretization of the right-hand side of (\ref{eqn:resolve_diri0}) corresponding to nodes $s^0_1,\dots,s^0_{N_0}.$ Further suppose that $U$ is the $P\times P$ matrix with entries $U_{ij} = \phi_i(s_j) \sqrt{w_j},$ where $\phi_1,\dots,\phi_P$ are the orthonormal functions on $(-\delta,\delta)$ spanning $,{\rm sgn}(s)|s|^\mu, |s|^\mu,$ $\mu=0,1/2-40$ on that interval. Let $w:(-\delta,\delta) \to \mathbb{R}^P$ be the vector-valued function defined by
\begin{align}
w(t) = (\phi_1(t),\phi_2(t),\dots,\phi_P(t)),
\end{align}
and $\tilde{\sigma} = (\sigma_1,\dots,\sigma_P)^T$ be the approximation to the solution in the vicinity of the corner obtained by solving the original system (\ref{eqn:adj_linalg}). Then $|(f(t)-h(t)) - w(t) U^{-1} A_0^T \tilde{\sigma}| = O(\epsilon)$ for all $t \in(-\delta,\delta).$ In particular, if $\tilde{f}_0$ is the vector of length $N_0$ with entries defined by
\begin{align}
(\tilde{f}_0)_i = w(s^0_i) U^{-1} A_0^T \tilde{\sigma}, \quad i=1,\dots,N_0,
\end{align}
then $\| \underline{f}_0 - \tilde{f}_0\| = O(\epsilon).$
\end{theorem}
\begin{proof}
We begin by observing that both $f$ and $h$ are analytic on the interval $(-\delta,\delta).$ In particular, they can be accurately interpolated using $\phi_1,\dots,\phi_P$ on the interval $(-\delta,\delta).$ Hence,
\begin{align}\label{eqn:approx_thm}
f(t)-h(t) \approx w(t) U^{-1} (f(s_1)-h(s_1),\dots,f(s_P)-h(s_P))^T.
\end{align}
A similar argument shows that $f(t)-h(t)$ is interpolable on $[-2\delta,-\delta]$ and $[\delta,2\delta].$
On the other hand, by construction,
\begin{align}\label{eqn:linalg_thm}
\underline{f} = A^T \begin{pmatrix}\tilde{\sigma}\\ \sigma_{P+{2M+1}}\\ \vdots\\ \sigma_N\end{pmatrix}.
\end{align}
Let $A_0$ be the $(P+2M)\times (P+2M)$ submatrix of $A$ corresponding to the first $P+2M$ rows and columns of $A,$ and $A_{\rm red}$ be the $P+2M \times (N-P-2M)$  submatrix of $A$ corresponding to selecting the first $P+2M$ rows of $A$ and all but the first $P+2M$ columns of $A.$ Using this notation, the first $P+2M$ rows of (\ref{eqn:linalg_thm}) can be re-written as
\begin{align}
\begin{pmatrix}f_1\\\vdots \\ f_{P+2M}\end{pmatrix}  = A_{\rm red}^T \begin{pmatrix} \sigma_{P+2M+1}\\ \vdots\\ \sigma_N\end{pmatrix} + A_0^T \tilde{\sigma}.
\end{align}
The first term on the right-hand side is $(h(s_1),\dots,h(s_{P+2M}))^T.$ Substituting this into the previous equation, we obtain
\begin{align}
(f(s_1)-h(s_1),\dots,f(s_{P+2M})-h(s_{P+2M}))^T = A_0^T \tilde{\sigma}.
\end{align}
The result follows by substituting the above equality into (\ref{eqn:approx_thm}).
\end{proof}

This can be iterated to obtain an interpolable approximation to the density on $L_0 = [-\delta,\delta].$ In particular, we consider the restriction of the exterior Neumann integral equation, as well as its, adjoint to the interval $I_0$ and $J_0.$ For the right-hand side we use the original right-hand side $f$ minus the contribution from the remainder of the domain. In particular, if we define
\begin{align}
h_1(s) = \sum_{x_i \in K_0,Q_0} k(s,x_i) \sigma_0^{(i)} \sqrt{w_i}
\end{align}
then $\sigma$ restricted to the interval $I_0\cup L_0 \cup J_0$ satisfies
\begin{align}\label{eqn:resolve_neum1}
-\sigma(s) + \int_{-\delta}^{\delta} k(s,t) \sigma(t)\,{\rm d} t = f(s)-h(s)-h_1(s), \quad -\delta \le s \le \delta.
\end{align}
The corresponding adjoint equation is given by 
\begin{align}\label{eqn:resolve_diri1}
-\sigma(s) + \int_{-\delta}^{\delta} k(t,s) \sigma(t)\,{\rm d} t = f(s)-h(s)-h_1(s), \quad -\delta \le s \le \delta.
\end{align}

Once again, we divide $L_0$ into three intervals $I_1,$ $L_1,$ and $J_1$ and discretize each interval as before. After solving the corresponding discretization of (\ref{eqn:resolve_neum1}) using the adjoint of the discretization of the integral operator appearing in (\ref{eqn:resolve_diri1}) we obtain a weak solution of $\sigma$ on the interval $I_0\cup L_0 \cup J_0$ which can be interpolated on $I_0,$ and $J_0$ to within precision $\epsilon.$

This process can be repeated an arbitrary number of times to yield a sequence of solutions $\underline{\sigma}_j,$ $j=0,1,2,\dots$ together with corresponding intervals $I_0,I_1,\dots$ and $J_0,J_1,\dots$ on which it can be interpolated.

Note that if $\bx$ is a point a distance $r$ away from the corner then after $J=1+\log_2 r/d$ such subdivisions $\bx$ will be at least twice the corner panel length away from the corner. Thus $K(\bx,\cdot)$ will be smooth when restricted to the corner panel $[-\delta/2^{J},\delta/2^J]$ and hence will be integrated accurately using the corner panel discretization nodes and weights.



\section{Numerical results \label{sec:num}}
\subsection{Accuracy}
In this section, we demonstrate the accuracy of the proposed numerical method (both in the {\it weak} sense 
described above, 
as well as in the classical sense after sufficiently many re-solves) on the triangular domain shown below. 
The reference solution for each of the examples is computed using a discretization with
a graded mesh in the vicinity of the corners, where the smallest panel at the corner is $2^{-200}$ times
the length of the first macroscopic panel away from the corner (see~\cref{fig:dom}). In these examples, the solutions are computed via dense linear solves.

\begin{figure}
\begin{center}
\includegraphics[width=0.7\linewidth]{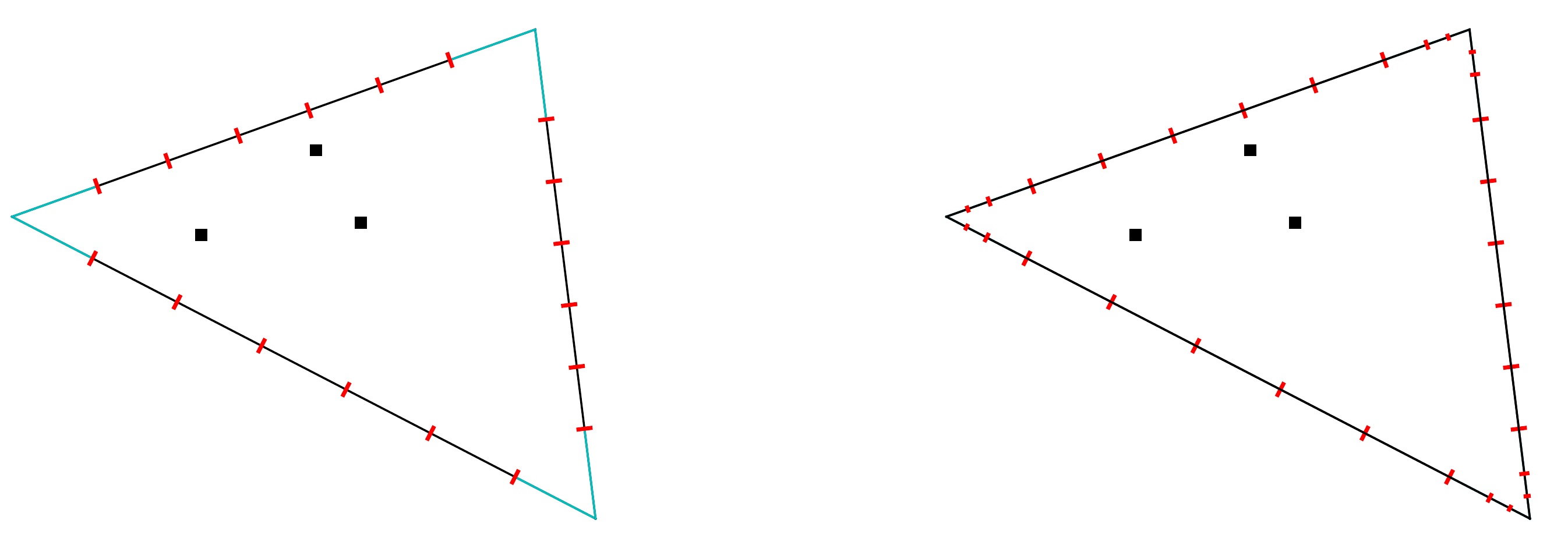}
\caption{Problem domain and panel discretization of the boundary. The discretization on the left is based on using the Dirichlet
discretization at corner panels (indicated by blue panels) discussed in~\cref{sec:disc_dir}, while the discretization on the right is a sample discretization with 2 levels of refinement in the vicinity of the corner. All the panels in black are discretized using scaled Gauss-Legendre nodes. The square ticks indicate location of the charges $\bx_{j}$ for defining the boundary data for the scattering problem.}
\label{fig:dom}
\end{center}
\end{figure}

\begin{rem}
Though $2^{-200}$ is significantly smaller than machine precision, the matrix entries corresponding to the
corner interactions can be computed accurately by translating the corners to the origin when 
computing interactions of nearby points.
\end{rem} 
\begin{rem}
Simple arguments from complex analysis show that when using graded meshes, in order to obtain full 
machine precision 
($\sim 1.11 \times 10^{-16}$) for solutions of the Neumann problem at any point in the interior 
at least $10^{-16}$ away from a corner, it suffices to choose the smallest panel (i.e. the size of the 
panel closest to the corner) to be of size $2^{-100}.$ However, resulting values of the density will not be 
accurate to machine precision at all nodes. In fact the quality
of the density deteriorates as one approaches the corner. Thus, in order to obtain accurate point values
of the density to machine precision at all points which are at least $2^{-100}$ away from the corner, we use a smallest
panel size of $2^{-200}$. 
\end{rem}

The potential at target locations which are sufficiently far from the boundary (i.e. at least one panel length away from 
every panel) is the inner product of the 
density with a smooth function and hence can be computed accurately without re-solving (see~\cref{rem:far-field-accuracy}). For a target 
location $\boldsymbol{y}$, we compute the potential via the formula,
\begin{equation}
\label{eq:farfieldpot}
u(\by) = \int_{\Gamma} G(\bx,\by) \sigma(\bx) dS_{\bx} \approx \sum_{i=1}^{N}  G(\gamma(s_{i}),\by) \sigma_{i} \sqrt{w_{i}}
\end{equation}
In~\cref{fig:scatteringtest}, we compute the error in the solution at target
locations for a scattering problem whose right hand side is given by a collection of three interior charges
\begin{equation}
\label{eq:scatboundarydata}
f(\bx) = -\nabla  \left(\sum_{j=1}^{3} \log |\bx - \bx_{j}|  \right)  \cdot \bnu(\bx)\, ,  
\end{equation}
where the locations $\bx_{j}$ are denoted by square dots in~\cref{fig:scatteringtest}. Note that the density $\sigma$ plotted
as a function of arclength goes to infinity at the corner vertices, indicating that the native Dirichlet discretization 
presented in~\cref{sec:disc_dir} wouldn't have sufficed. However, the potential in the volume is accurate to 14 digits at
target locations away from the boundary. 
\begin{figure}
\begin{center}
\includegraphics[width=\linewidth]{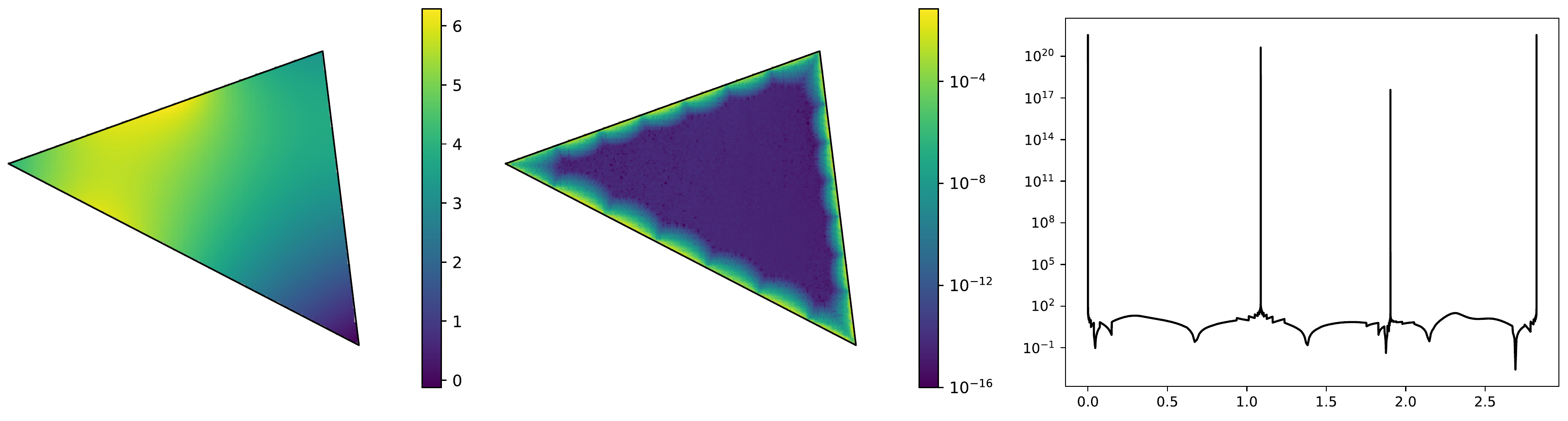}
\caption{Left panel: Solution to Neumann problem with data given by~\cref{eq:scatboundarydata}, center panel: error in computing the potential in the formula using the underlying smooth quadrature~\cref{eq:farfieldpot}, and on the right the density $\sigma$ as a function of arc-length}.
\label{fig:scatteringtest}
\end{center}
\end{figure}

Another example of a ``weak quantity'' is the polarization tensor associated with a domain. This requires the solution of the
exterior problems with boundary data $f_{1} = \bnu_{1}$ or $f_{2} = \bnu_{2}$. Let $\sigma_{1}$ and $\sigma_{2}$ denote
the corresponding solutions. The polarization tensor can be expressed in terms of the solutions $\sigma_{1}$ and $\sigma_{2}$
as
\begin{equation}
P = \begin{bmatrix}
\int_{\Gamma} x_{1} \sigma_{1}(\bx) dS_{\bx} & \int_{\Gamma} x_{2} \sigma_{1}(\bx) dS_{\bx} \\
\int_{\Gamma} x_{1} \sigma_{2}(\bx) dS_{\bx} & \int_{\Gamma} x_{2} \sigma_{2}(\bx) dS_{\bx} 
\end{bmatrix}
\end{equation}
The polarization tensor as computed by the reference solution, and the error in computation using the adjoint discretization
are given by
\begin{equation}
P = \begin{bmatrix}
-0.823641009939200 & -0.139714174784448 \\
 -0.139714174784448 &  -1.1421444446470226
\end{bmatrix} \, , \quad \text{Error} = 
\begin{bmatrix} 
2.3 \times 10^{-15} & 7.9 \times 10^{-16} \\
1.3 \times. 10^{-14} & 2.7 \times 10^{-15}
\end{bmatrix}
\end{equation}

In order to demonstrate the accuracy of the corner re-solving approach in obtaining the true density at the corner panels, we  apply the procedure
discussed in~\cref{sec:resolve} iteratively, and compare the obtained density with the reference density after 20,40,60, and
80 iterations of resolves in the vicinity of one of the corners. The reference density and the errors are shown in ~\cref{fig:dens-error}. Furthermore, to highlight the need for special purpose discretizations in the vicinity of corners in the adjoint discretization, we also compare the solution computed using a graded mesh in the vicinity of corners, where the size of the smallest panels for both discretizations are equal. 

\begin{figure}
\begin{center}
\includegraphics[width=\linewidth]{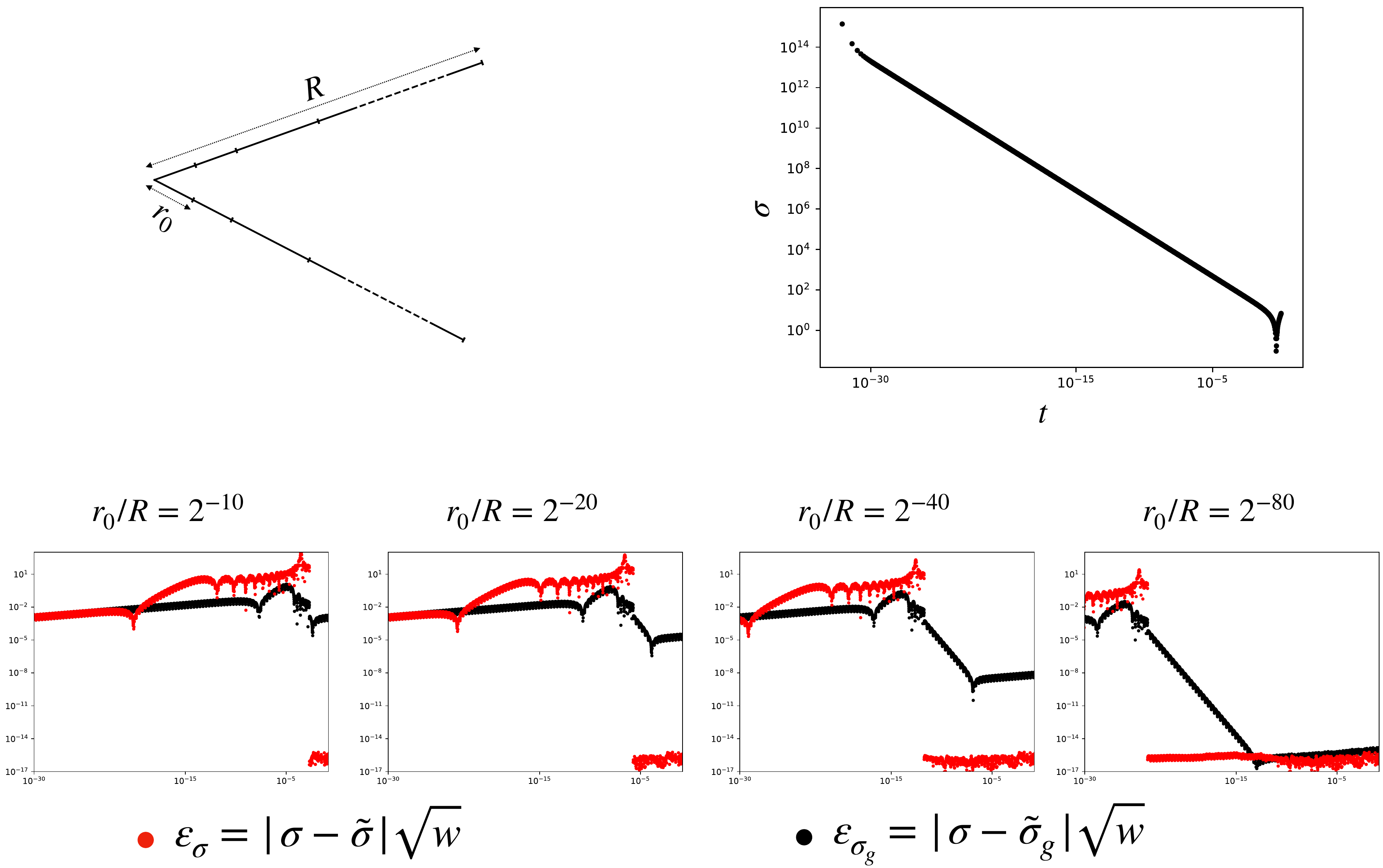}
\caption{Top row: (left) Illustrative mesh used for iteratively computing the solution in the vicinity of a corner, (right) the density in the vicinity of one of the corner panels. Bottom row: error in computing the density, where $\tilde{\sigma}$ denotes the density computed using special purpose discretizations at corner panels, and $\tilde{\sigma_{g}}$ denotes the density using a graded mesh with the smallest panel equal to the length of the smallest panel after the iterative resolve procedure. The errors are scaled by square roots of the quadrature weights.    }.
\label{fig:dens-error}
\end{center}
\end{figure}

After re-solving the density, the solution is evaluated on a tensor product polar grid, where the grid
is exponentially spaced in the radial direction and equispaced in the angular direction. For evaluation points (targets)
close to panels which are not at the corner, we use adaptive integration in order to resolve the near-singular behavior of
the kernel for accurate computation of the integrals. For target locations close to the corner panel, since we do not have
the capability to interpolate the density, we use the underlying smooth quadrature rules for computing their contribution. 
The reference solution and the errors are demonstrated in~\cref{fig:vol-plot}.
\begin{figure}
\begin{center}
\includegraphics[width=\linewidth]{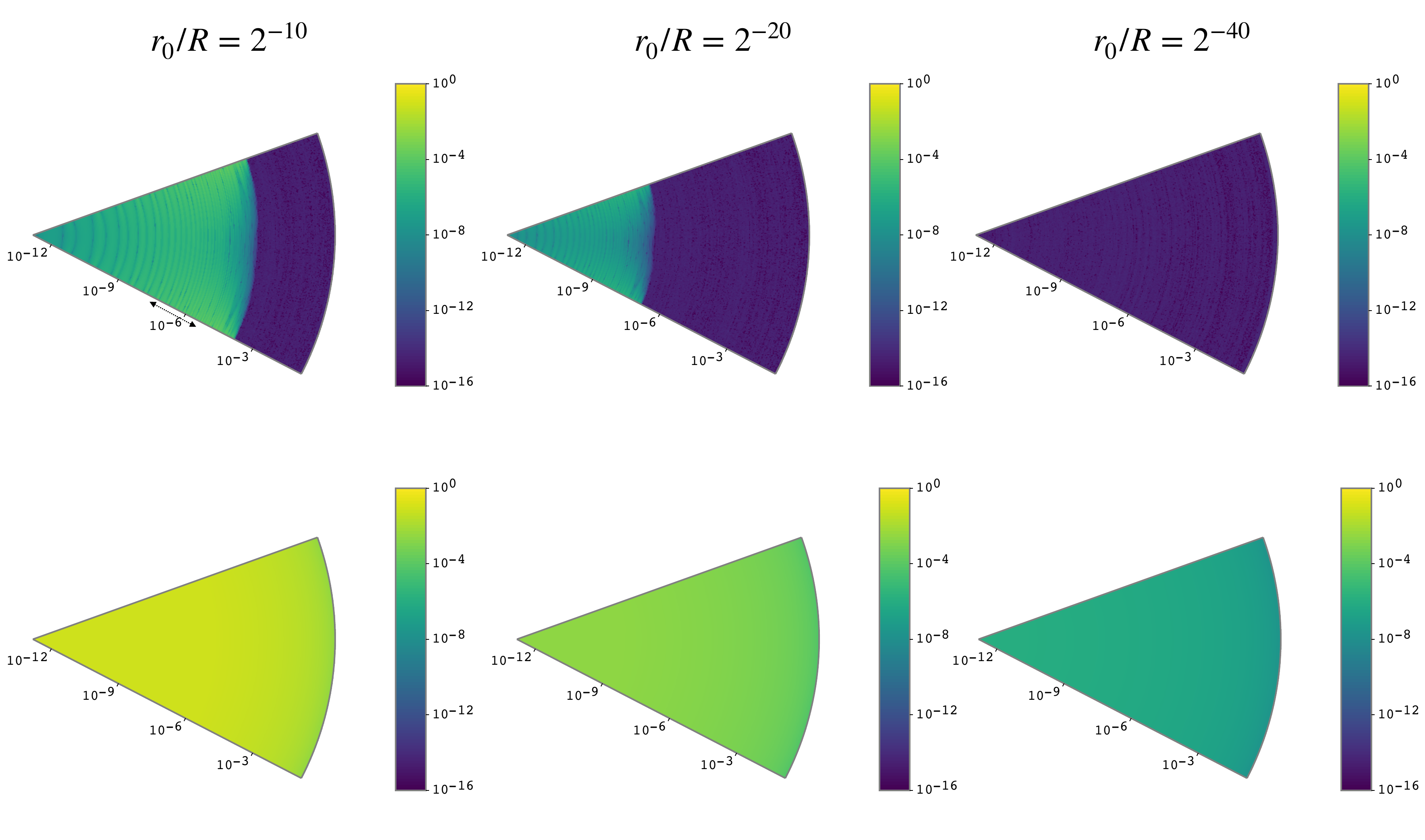}
\caption{Top row: Solution on the volume after 10, 20, and 40 iterations of re-solve. The solution is computed on a tensor product polar grid, where the evaluation points (targets) are exponentially spaced in the radial direction. The closest target location is approximately $10^{-13}$ away from the corner. Near quadrature is handled via adaptive integration except for the corner panel where the smooth quadrature weights are used. Bottom row: analogous results where the solution is computed using a graded mesh.   }.
\label{fig:vol-plot}
\end{center}
\end{figure}

\subsection{Performance}
In this section, we demonstrate the performance of the solver by solving a scattering problem in the
exterior of a ``broken wheel'' region. The boundary data is given by
\begin{equation}
f(x) = \nabla  \sum_{j=1}^{57} c_{j} \log| \bx - \bx_{j}| \cdot \bnu(\bx) \, , 
\end{equation}
where there is one $\bx_{j}$ located in each of the spokes, one of the $\bx_{j}$ is in the central disc, and the remaining $50$ $\bx_{j}$ are chosen randomly in the exterior of the bounding disc containing the domain. The strengths $c_{j}$ are chosen such that they average to $0$. The domain contains $108$ corners, was discretized using $22240$ nodes and required $105$ iterations to converge to a residue of $10^{-15}$. The matrix at each iteration was applied using an FMM whose tolerance was also set to $10^{-15}$. The solution was computed in $15$ secs, and plotted at a $500\times 500$ grid of targets in $6.5$ secs. All of the results have been computed on a single core on a Macintosh machine with Intel core i5 2.3GHz processors. In~\cref{fig:magnetron}, we plot the scattered field, the boundary data, and the computed density.

\begin{figure}
\begin{center}
\includegraphics[width=\linewidth]{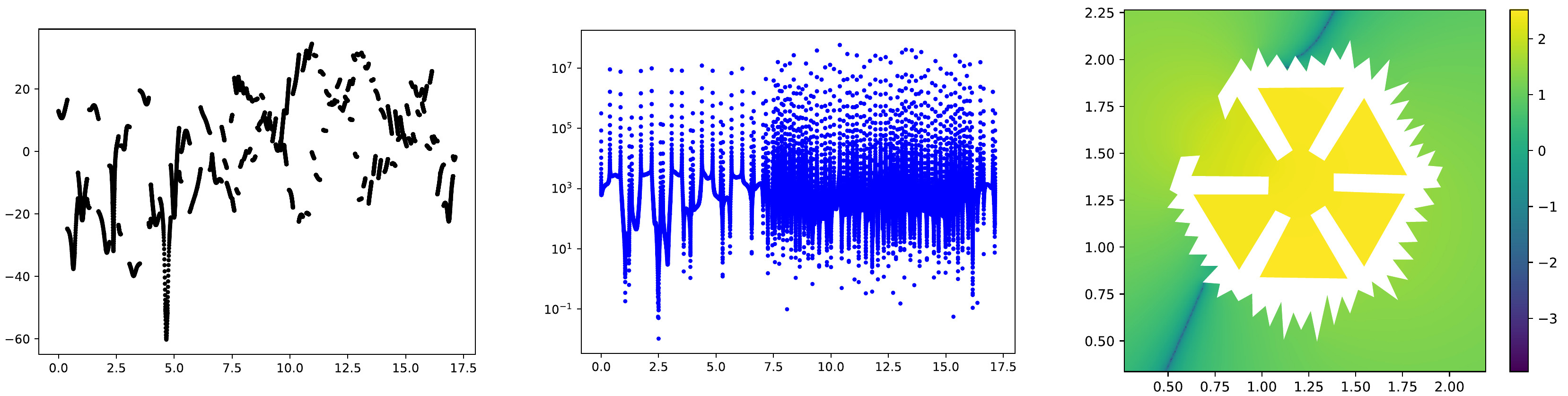}
\caption{(left): Boundary data as a function of arclength, (center): absolute value of density as a function of arclength, and 
(right): $\log_{10}$ of the absolute value of the solution in the volume computed using an FMM.}
\label{fig:magnetron}
\end{center}
\end{figure}



\section{Conclusion and future work}

In this paper we described a method for obtaining solutions to Laplace's equation with Neumann boundary conditions 
on polygonal domains given an accurate discretization of a corresponding Dirichlet problem. 
The resulting solutions are accurate in a 
``weak sense'', allowing evaluation of the solution at points which are located sufficiently far from the boundary of the 
domain. We then presented a method for using these ``weak solutions'' to obtain accurate solutions to the Neumann problem 
in an $L^\infty$-sense arbitrarily close to the corner in a computationally efficient manner. 

Though the present paper treats only Laplace's equation for polygonal domains, the method shown here extends 
much more broadly. In particular, the approach easily extends to accommodate curved boundaries. Moreover, in addition 
to Laplace's equation, this approach can be easily adapted to solve the Helmholtz equation and the biharmonic equation with analogous boundary conditions for which the nature of singularities of corresponding integral equations have been analyzed~\cite{serkhrachh,serkhpnas}. A manuscript detailing this extension is currently in preparation.


\section{Acknowledgments}
J. Hoskins was supported in part by AFOSR FA9550-16-1-0175 and by the ONR (award no. N00014-14-1-0797).
The authors would like to thank Alex Barnett, Leslie Greengard, Michael O'Neil, and Vladimir Rokhlin for many useful discussions, and Jeremy Magland for providing sector plotting tools.

\appendix
\section{Approximation of data on corner panels for re-solve}
Here we give explicit bounds for the rate of convergence of the contribution of the rest of the boundary to a corner panel. In particular, given a polygonal domain with boundary $\Gamma,$ let $\bx$ denote a vertex of $\Gamma$ and $C = \Gamma \cap B_r(\bx),$ where $B_r(\bx)$ is the ball of radius of $r$ centered at $\bx.$ We choose $r$ so that $\Gamma \cap B_{2r}(\bx)$ corresponds to a wedge with internal angle $\pi \alpha$ and side lengths $2r.$ 


\begin{theorem}
Let $\Gamma$ be the boundary of a polygon and $\bx$ be a vertex. Let $r>0$ be a real number such that $\Gamma \cap B_{2r}(\bx)$ corresponds to a wedge with internal angle $\pi \alpha$ and side lengths $2r,$ where $B_R(\bx)$ denotes a ball of radius $R$ centered at $\bx.$ Let $L$ denote the length of $\Gamma$ and $\gamma: [-L/2,L/2]\to \Gamma$ be an arclength counterclockwise parameterization of $\Gamma$ such that $\gamma(0) = \bx.$ Finally, for any $f \in L^2(\Gamma),$ let $H:[0,r] \to \mathbb{R}$ be the function defined by
\begin{align}
H(t) = \int_{\Gamma\backslash B_{2r}(\bx)} K(\gamma(t),\gamma(s)) f(s)\,{\rm d}s.
\end{align}
Then $H$ is analytic in a neighborhood of $0$ with Taylor series coefficients $\{ a_n\}$ satisfying
\begin{align}
|a_n| \le \frac{\sqrt{L}}{2^nr^{n+1}} \|f\|_{L_2(\Gamma \backslash B_{2r}(\bx))}.
\end{align}
\end{theorem}
\begin{proof}
Without loss of generality we can assume that $\Gamma$ is shifted, oriented and parameterized so that $\bx = \gamma(0)= 0$ and the leg of the wedge corresponding to positive $t$ is oriented along the positive $x$ axis. Then
\begin{align}
H(t) = \int_{\Gamma\backslash B_{2r}(\bx)} K(\gamma(t),\gamma(s)) f(s)\,{\rm d}s = \int_{\Gamma\backslash B_{2r}(\bx)} \frac{y(s)}{(t-x(s))^2+y(s)^2}f(s)\,{\rm d}s.
\end{align}
Since $\|\gamma(s)-\gamma(t)\| >2r$ it follows that
\begin{align}
H(t) =\int _{\Gamma\backslash B_{2r}(\bx)}\sum {t^n i}\left(\frac{1}{(x(s)+i y(s))^{n+1}}-\frac{1}{(x(s)-i y(s))^{n+1}} \right) f(s)\,{\rm d}s.
\end{align}
In particular, $H$ has a Taylor series about $t=0,$
\begin{align}
\sum_{n=0}^\infty c_n \left(\frac{t}{2r}\right)^n
\end{align}
where
\begin{align}
|c_n| \le  \frac{\sqrt{L}}{r} \|f\|_{L_2(\Gamma \backslash B_{2r}(\bx))}.
\end{align}
\end{proof}

\section{Strong approximation of density away from corner panels}
\label{sec:appb}
In this section, we demonstrate that for a panel which is sufficiently far from the corner and  discretized using $M$ Gauss-Legendre nodes, the density computed using the adjoint of a Dirichlet discretization can be interpolated 
accurately at any point on the panel. As before let $\Omega$ denote a polygonal domain with boundary $\Gamma$. Let
$L$ denote the length of the boundary, and let $\gamma: [0,L]\to \mathbb{R}^2$ denote an arc-length parameterization of the boundary.
Assume that the discretization satisfies the following assumptions:
\begin{enumerate}
\item All panels which are not at a corner, are separated from the closest corner by at least their panel length.
\item Let $E_{\rho}(\Gamma_{i})$ denote the Bernstein $\rho-$ellipse (see \cite{tref_approx}) corresponding to the panel $\Gamma_{i}$, and let $\rho_{i}$ 
be such that $E_{\rho_{i}}(\Gamma_{i})$ does not intersect $\Gamma \setminus S_{\Gamma_{i}}$, where 
$S_{\Gamma_{i}}$ is the edge containing $\Gamma_{i}$ (see~\cref{fig:illus-cons}). Let $\rho_{0} = \min_{i} \rho_{i} >1$.
\item All panels which are not adjacent to a panel at the corner, are separated from the corner by $2r_{c}$ where $r_{c}$ 
denotes the length of the panel at the corner.
\end{enumerate}
Under these assumptions, it can be shown that the accuracy of computing the
Legendre coefficients of the density (at panels which are not adjacent to a corner panel) for the Neumann problem using the adjoint discretization is related to the accuracy 
in the computation of the solution to an associated Dirichlet problem.
\begin{figure}
\begin{center}
\includegraphics[width=0.3\linewidth]{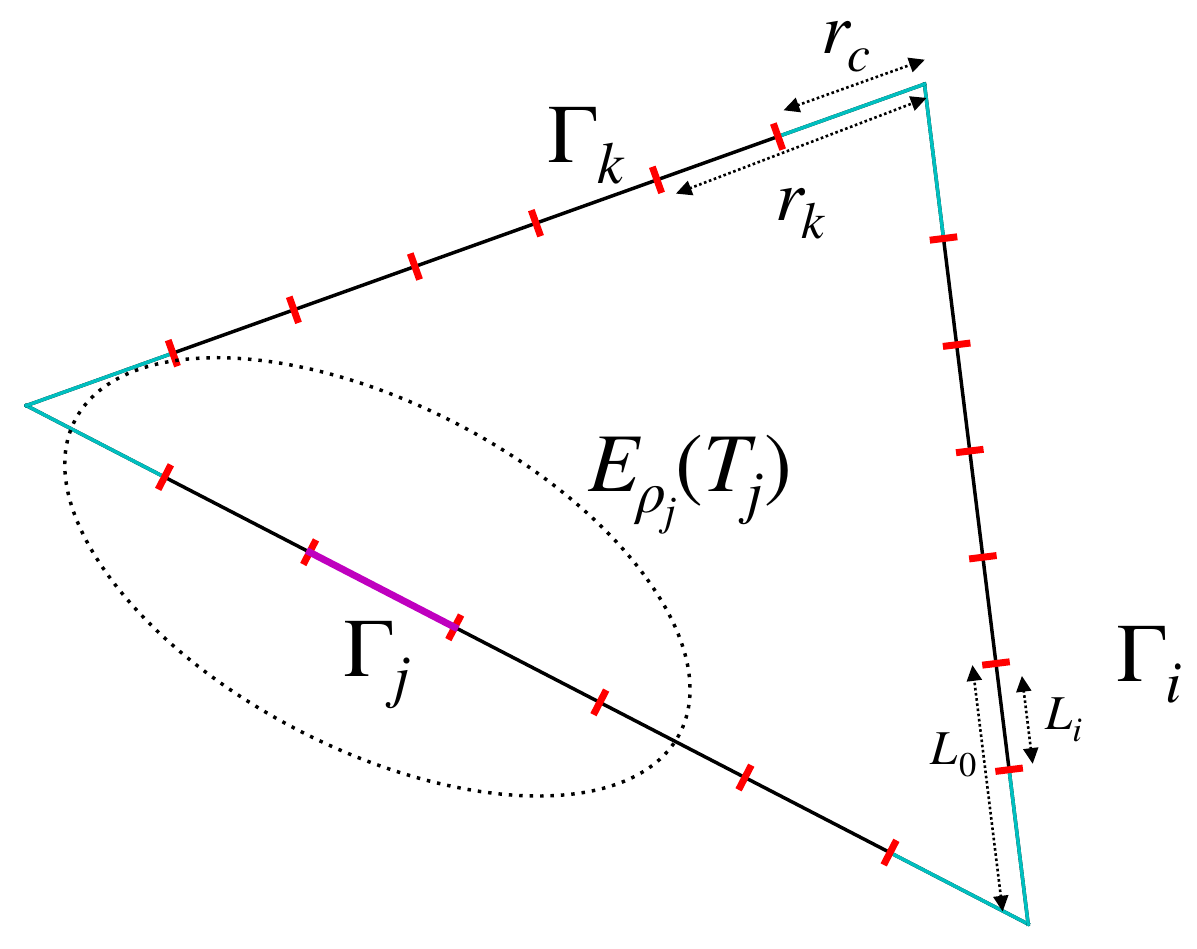}
\caption{Illustrative figure for demonstrating constraints required on the geometry discretization. The panel $\Gamma_{i}$ satisfies restriction 1 if $L_{i}/L_{0}>1$. Panel $\Gamma_{j}$ illustrates the largest Bernstein ellipse which intersects the other edges, and Panel $\Gamma_{k}$ illustrates restriction 3 if $r_{k}>2r_{c}$. }
\label{fig:illus-cons}
\end{center}
\end{figure}

Let $A$ denote the operator corresponding to interior Dirichlet problem using a double layer potential. Let $g$ denote the
right hand side for the Neumann problem, let $\sigma$ denote the corresponding solution. Let $f$ be a Legendre polynomial of
degree $n$ scaled to the panel $\gamma([s_{1},s_{2}])$ and $0$ everywhere else. Then
\begin{equation}
\left<\sigma, f \right> = \left< (A^{T})^{-1} g, f\right> = \left< g, A^{-1} f \right> = \langle g, \sigma_{f}\rangle \, ,
\end{equation}
where $\sigma_{f}$ is the solution of the interior Dirichlet problem with boundary data $f$ using a double layer potential.

Using~\cref{lem_adjm}, the above statement implies that the error in computing the $M$ Legendre coefficients of
the density for the Neumann problem is the same as the error in computing the solution of a Dirichlet problem with 
data given by a Legendre polynomial on the same panel. 

Let $V$ denote the collection of corner points in parameter space $[0,L]$, i.e. $a \in V$ if $\gamma(a)$ is a corner vertex.
Recall that $\sigma_{f}$ denotes the solution of the Dirichlet problem with boundary data $f$, i.e., 
$\sigma_{f}$ satisfies
\begin{equation}
-\frac{\sigma_{f}(s)}{2}   + \int_{0}^{L} k(s,t) \sigma_{f}(t) dt  = f(s) \,  \quad s\in[0,L]\setminus{V}.
\end{equation}
 Then $\sigma_{f} = -2f + \tilde{\sigma}$ where $\tilde{\sigma}$ satisfies the integral equation
 \begin{equation}
-\frac{\tilde{\sigma}}{2}+ \int_{0}^{L} k(s,t) \tilde{\sigma}(t) dt = -2\int_{s_{1}}^{s_{2}} k(s,t) f(t) dt \, , \quad s\in[0,L] \setminus V \, ,
 \end{equation}
i.e, $\tilde{\sigma}$ is the solution of the Dirichlet with data $\tilde{f}$ given by
\begin{equation}
\tilde{f}(s) = -2\int_{s_{1}}^{s_{2}} k(s,t) f(t) dt \, , s \in [0,L] \setminus V \, .
\end{equation}
There are two concerns which must be addressed. First, the accuracy of computing $\tilde{f}(s)$ 
for any point $s \in [0,L]$ using an $M$ point Gauss-Legendre quadrature on $[s_{1},s_{2}]$, and secondly, 
the resolution of the function $\tilde{f}(s)$ on the given discretization of the boundary. 

For any $s$ which is contained on the same segment as $\gamma([s_{1},s_{2}])$, the kernel $k(s,t)$ is identically $0$.
Thus the boundary data $\tilde{f}(s)=0$  on the same edge as the panel $\gamma([s_{1},s_{2}])$. 
We further observe that $f(t)$ is an entire function when extended to the complex plane, since it is a Legendre polynomial.
Moreover, the nearest singularity of the function $k(s,t)$ in the complex plane as a function of $t$ is at $\gamma(s)$.
From assumption 2, it follows that
the error in computing $\tilde{f}(s)$ using an $M$ point Gauss-Legendre rule is bounded by $C \rho_{0}^{-M}$, where
the constant $C$ is related to the smoothness of $k(s,t)$ as a function of $t \in \mathbb{C}.$
Thus the function $\tilde{f}(s)$ can be computed to any desired precision by increasing the order of quadrature nodes
used to compute the integrals.

With regards to the resolution of the of the function $\tilde{f}(s)$ on the given discretization of the boundary, 
we note that the closest singularity of the function $\tilde{f}(s)$ when restricted to a panel away from the corner
and not on the same edge as $\gamma([s_{1},s_{2}])$ is the closest point on the panel $\gamma([s_{1},s_{2}])$. 
However, by assumption 2, the error in resolving the function $\tilde{f}(s)$ using an $M$ point Gauss-Legendre
or Chebyshev panel is bounded by $C \rho_{0}^{-M}$. Note that the behavior of $\tilde{f}(s)$ in the complex plane is related
to the behavior of $k(s,t)$ in the complex $s$ plane and hence the constant $C$ is $O(1)$.
For the panels, at the corner, based on the proof in Appendix A, the error in resolving the function $\tilde{f}(s)$ when
truncated to a Taylor series of order $N$ is less than $C 2^{-N}$, since all points on $\gamma([s_{1},s_{2}])$ are 
well-separated from corners by twice the panel length $r_{c}$. Thus, by making the panels small enough, $\rho$ can be increased arbitrarily to obtain desired tolerances on the boundary data $\tilde{f}(s)$ on the corresponding discretization of the boundary. 
 
 Thus, the boundary data $\tilde{f}(s)$ is piecewise analytic , which can be approximated to any desired tolerance by appropriately reducing the panel sizes. This is the precise setup for which the discretization of the Dirichlet problem is designed to obtain accurate solutions to the density $\tilde{\sigma}$. 
 

\end{document}